\documentclass[12pt]{article}
\usepackage{amsmath,amssymb,amsthm}
\def\Cmath{\mathbb{C}}

\def\Rmath{\mathbb{R}}

\newtheorem{lemma}{Lemma}
\newtheorem{remark}{Remark}

\title{Two examples about zero torsion linear maps on Lie algebras
\thanks{\textit{Math. Subj. Class. [2000]}  : 17B30.
 \textit{Key words} : Complex structures, $CR$-structures, zero torsion, Heisenberg Lie algebra, $\mathfrak{sl}(2,\Rmath)$.
}
}
\author{L. Magnin \\
Institut Math\'{e}matique de Bourgogne
\thanks{UMR CNRS 5584, Universit\'{e} de Bourgogne, BP 47870, 21078 Dijon Cedex, France.}
\\
\texttt{magnin@u-bourgogne.fr}
}
\date{\today}
\begin{document}
\maketitle
\begin{abstract}
The question of whether or not any zero torsion linear map on a non abelian real Lie algebra
$\mathfrak{g}$
is necessarily an extension of some $CR$-structure is considered and answered in
the negative. Two examples are provided, one in the negative and one in the positive.
In both cases, the computation up to equivalence of all zero torsion linear maps on
$\mathfrak{g}$ is used for an explicit description of the equivalence classes of
integrable complex structures on
$\mathfrak{g} \times \mathfrak{g}.$
\end{abstract}
\section{Introduction.}
Given a real Lie algebra $\mathfrak{g},$ the determination up to equivalence of
zero torsion linear maps from $\mathfrak{g}$ to $\mathfrak{g}$ plays an important
role in the computation of complex structures on direct products involving
$\mathfrak{g}$
(\cite{u2revisited}). In the present note, we consider the question of
whether or not
any such zero torsion linear map
for non abelian $\mathfrak{g}$ is necessarily an extension of some
$CR$-structure.
We answer the question in the negative by computing (up to equivalence) all zero torsion linear maps
from the real 3-dimensional Heisenberg Lie algebra $\mathfrak{n}$ into itself. The result is then used to exhibit
a complete set of
representatives  of equivalence classes of  complex structures on $\mathfrak{n}\times \mathfrak{n}$.
We also compute
all zero torsion linear maps
on
 $\mathfrak{sl}(2,\Rmath).$ In that case they are extensions of $CR$-structures. We deduce
a complete set of
representatives  of equivalence classes of  complex structures on
 $\mathfrak{sl}(2,\Rmath) \times  \mathfrak{sl}(2,\Rmath).$
\section{Preliminaries.}
Let $G_0$ be a connected finite dimensional real Lie group, with
Lie algebra $\mathfrak{g}.$
A linear map  $J \, : \, \mathfrak{g} \rightarrow  \mathfrak{g}$
is said to have \textit{zero torsion} if it satisfies the condition
\begin{equation}
\label{Jnotorsion}
 [{J} X, {J}Y]-[X,Y]-{J}[{J}X,Y]-J[X,{J}Y] =0 \quad \forall X,Y \in \mathfrak{g}.
 \end{equation}
If $J$ has zero torsion and  satisfies  in addition $J^2=-1,$
 $J$ is an (integrable) complex structure on
$\mathfrak{g}.$
That means that $G_0$
can be given the structure of a complex manifold with the same underlying
real structure and such that the canonical complex structure on $G_0$ is the left invariant almost
complex structure
$\hat{J}$ associated to $J$
(For more details, see \cite{artmagnin1}).
To any (integrable) complex structure $J$ is associated
the complex subalgebra
$\mathfrak{m}=\left\{ \tilde{X} := X - i J X \, ; X \in \mathfrak{g} \right\}$
of the complexification $\mathfrak{g}_\Cmath$ of $\mathfrak{g}.$
In that way, (integrable) complex structures can be identified with
complex subalgebras $\mathfrak{m} $ of
$\mathfrak{g}_\Cmath$  such that
$\mathfrak{g}_\Cmath =\mathfrak{m} \oplus \bar{\mathfrak{m}}$, bar denoting conjugation.
$J$ is said to be abelian if $\mathfrak{m}$ is.
When computing the matrices of the zero torsion maps in some
fixed basis
 $(x_j)_{ 1 \leqslant j \leqslant n}$
of $\mathfrak{g}$,
we will denote by
$ij|k$ ($1 \leqslant i,j,k \leqslant n$)
the torsion equation
obtained by projecting on $x_k$ the equation
(\ref{Jnotorsion}) with $X=x_i, Y=x_j.$
The automorphism group
$\text{Aut } \mathfrak{g}$
of $\mathfrak{g}$ acts on the set
of all  zero torsion linear maps and on the set
of all complex structures on  $\mathfrak{g}$ by
$J \mapsto \Phi \circ J \circ \Phi^{-1} \quad \forall \Phi \in
\text{Aut } \mathfrak{g}.$
Two $J,J^{\prime}$
on   $\mathfrak{g}$  are said to be \textit{equivalent}
(notation: $J \equiv J^{\prime}$)
if they are on the same
$\text{Aut } \mathfrak{g}$ orbit.
For complex structures and   simply connected $G_0,$
this amounts to the existence of an $f \in \text{Aut } G_0$ such that
$f : (G_0,J) \rightarrow  (G_0,J^{\prime})$ is biholomorphic.
\section{Case of
 $\mathfrak{sl}(2,\Rmath).$}
Let $G=SL(2,\Rmath)$ denote the  Lie group of
real  $2\times 2$ matrices with determinant 1
\begin{equation}
\label{sigma}
\sigma = \begin{pmatrix} a&b \\ c&d  \end{pmatrix} \quad , \quad ad-bc=1.
\end{equation}
Its Lie algebra
 $\mathfrak{g} =\mathfrak{sl}(2,\Rmath)$
 consists of the zero trace
real $2\times 2$ matrices
$$X = \begin{pmatrix} x&y \\ z&-x  \end{pmatrix}= xH+yX_+  +zX_- $$
with  basis
$H = \left(\begin{smallmatrix} 1&0 \\ 0&-1  \end{smallmatrix}\right)$,
$X_+ = \left(\begin{smallmatrix} 0&1 \\ 0&0 \end{smallmatrix}\right)$,
$X_- = \left(\begin{smallmatrix} 0&0 \\ -1&0 \end{smallmatrix}\right)$
and commutation relations
\begin{equation}
\label{relationssl2}
[H,X_+]=2X_+, \; [H,X_-]=-2X_-, \; [X_+,X_-]=H.
\end{equation}
Beside the basis $(H,X_+,X_-),$  we shall also make use of  the basis
$(Y_1,Y_2,Y_3)$
where
$Y_1=\frac{1}{2} H,$
$Y_2=\frac{1}{2} (X_+ -X_-),$
$Y_3=\frac{1}{2} (X_+  + X_-),$ with
 commutation relations
\begin{equation}
\label{relationssl2bis}
[Y_1,Y_2]=Y_3,
[Y_1,Y_3]=Y_2,
[Y_2,Y_3]=Y_1.
\end{equation}
The adjoint representation of $G$ on $\mathfrak{g}$ is given by
$\text{Ad}(\sigma) X = \sigma X \sigma^{-1}.$
The matrix $\Phi$ of $\text{Ad}(\sigma)$
($\sigma$ as in (\ref{sigma}))
in the basis $(H,X_+,X_-)$ is
\begin{equation}
\Phi =
\begin{pmatrix}
1+2bc&-ac&bd\\
-2ab&a^2&-b^2\\
2cd&-c^2&d^2
\end{pmatrix}.
\end{equation}
The adjoint group $\text{Ad}(G)$ is the identity component of
$\text{Aut } \mathfrak{g}$
and one has
\begin{equation}
\label{Aut(sl2)}
\text{Aut } \mathfrak{g}
 = \text{Ad}(G) \cup  \Psi_0 \text{Ad}(G) \quad , \quad
 \Psi_0= \text{diag}(1,-1,-1).
\end{equation}
The adjoint action of $G$ on $\mathfrak{g}$ preserves the form $x^2 +yz.$
The orbits are :
\\ (i) the trivial orbit $\{0\};$
\\ (ii) the upper sheet $z>0$ of the cone $x^2+yz=0$ (orbit of $X_-$);
\\ (iii) the lower sheet $z<0$ of  the cone $x^2+yz=0$ (orbit of $-X_-$);
\\ (iv) for all $s >0$ the one-sheet hyperboloid $x^2+yz=s^2$ (orbit of $s H$);
\\ (v)
for all $s >0$
the upper sheet $z>0$ of the hyperboloid $x^2+yz=-s^2$
(orbit of $s(-X_+ +X_-$));
\\ (vi)
for all $s >0$
the lower sheet $z<0$ of the hyperboloid $x^2+yz=-s^2$ (orbit of $s(X_+ -X_-$)).
\\
The orbits of $\mathfrak{g}$ under the whole
$\text{Aut } \mathfrak{g}$
are, beside $\{0\}$:
  \\ (I) the cone $x^2+yz=0$ (orbit of $X_-$);
\\ (II) the one-sheet hyperboloid $x^2+yz=s^2$ (orbit of $s H$)
($s >0$);
\\ (III) the two-sheet hyperboloid $x^2+yz=-s^2$ (orbit of $s(X_+ -X_-$))
($s >0$).

\begin{lemma}
\label{sl2notorsion}
 Let $\mathfrak{g}=
 \mathfrak{sl}(2,\Rmath),$ and
  $J \, : \,
 \mathfrak{g}
\rightarrow
 \mathfrak{g}
$ any linear map.
 $J$ has zero torsion
if and only if it is equivalent to
the endomorphism defined  in the basis
$(Y_1,Y_2,Y_3)$
(resp.
$(H,X_+,X_-)$)
by
\begin{equation}
\label{Jsl2notorsionbis}
J_*(\lambda)=\begin{pmatrix}
0&0&-1\\
0&\lambda&0\\
1&0&0
\end{pmatrix}
\quad , \quad \lambda \in \Rmath \quad ,
\end{equation}
 $J_*(\lambda) \not \equiv  J_*(\mu)$ for $\lambda \neq \mu$

 (resp.
\begin{equation}
\label{Jsl2notorsion}
J(\alpha)=\begin{pmatrix}
0&-\frac{1}{2}&-\frac{1}{2}\\
1&\alpha&-\alpha\\
1&-\alpha&\alpha
\end{pmatrix}
\quad , \quad \alpha \in \Rmath  \quad ,
\end{equation}
 $J(\alpha) \not \equiv  J(\beta)$ for $\alpha \neq \beta$).
\end{lemma}
\begin{proof}
Let $J = (\xi^i_j)_{1 \leqslant i,j \leqslant 3}$
in the basis $(H,X_+,X_-).$
The 9 torsion equations are
in the basis $(H,X_+,X_-)$:
\begin{eqnarray*}
12|1
&\quad
2 (\xi^2_2 + \xi^1_1) \xi^1_2 + (\xi^2_2 - \xi^1_1) \xi^3_1 - (\xi^2_1
+ 2 \xi^1_3) \xi^3_2=0,\\
12|2
&\quad
 2 (\xi^2_1 \xi^1_2 + 1 + (\xi^2_2)^2) - \xi^3_1 \xi^2_1 - 2 \xi^3_2 \xi^2
_3=0,\\
12|3
&\quad
  (\xi^3_1 + 2 \xi^1_2) \xi^3_1 - 2 (\xi^2_2 + 2 \xi^1_1) \xi^3_2 + 2 \xi^3_3 \xi^3_2=0,\\
13|1
&\quad
(\xi^2_1 - 2 \xi^1_3) \xi^1_1 + 2 \xi^2_3 \xi^1_2 + \xi^3_1 \xi^2_3 - (
\xi^2_1 + 2 \xi^1_3) \xi^3_3=0,\\
13|2
&\quad
2 (\xi^2_2 - 2 \xi^1_1) \xi^2_3 + (\xi^2_1 + 2 \xi^1_3) \xi^2_1 - 2 \xi^
3_3 \xi^2_3=0,\\
13|3
&\quad
\xi^3_1 \xi^2_1 - 2 \xi^3_1 \xi^1_3 - 2 + 2 \xi^3_2 \xi^2_3 - 2 (\xi^3_3)^2=0,\\
23|1
&\quad
 4 \xi^1_3 \xi^1_2 - 1 - \xi^2_2 \xi^1_1 - \xi^3_2 \xi^2_3 + (\xi^2_2 - \xi^1_1) \xi^3_3=0,\\
23|2
&\quad
 4 \xi^2_3 \xi^1_2 - (\xi^2_2 + \xi^3_3) \xi^2_1=0,\\
23|3
&\quad
4 \xi^3_2 \xi^1_3 - (\xi^2_2 + \xi^3_3) \xi^3_1=0.
\end{eqnarray*}
$J$ has at least one real eigenvalue $\lambda.$ Let $v \in \mathfrak{g},$ $v \neq 0,$ an
eigenvector  associated to $\lambda.$
From the classification of the
$\text{Aut } \mathfrak{g}$ orbits of
$\mathfrak{g},$ we then get 3 cases according to whether $v$ is on the orbit (I),(II),(III)
(in the cases (II), (III) one may choose $v$ so that $s=1$).
\\Case 1. There exists
$\varphi \in \text{Aut } \mathfrak{g}$ such that $v= \varphi(X_-).$ Then,
replacing $J$ by $\varphi^{-1} J \varphi,$
we may suppose
$\xi^1_3=\xi^2_3=0.$
That case is impossible from $13|2$ and $13|3.$
\\Case 2. There exists
$\varphi \in \text{Aut } \mathfrak{g}$ such that $v= \varphi(H).$ Then we may suppose
$\xi^2_1=\xi^3_1=0.$ Then from $12|2$, $\xi^2_3\xi^3_2 \neq 0,$ and $23|2$, $23|3$
yield $\xi^1_2=\xi^1_3=0.$
Then $12|3$ and $13|2$ successively give  $\xi^3_3=\xi^2_2+2\xi^1_1$ and
 $\xi^1_1=0.$ Now $12|2$ and $23|1$ read resp. $-\xi^2_3\xi^3_2 + (\xi^2_2)^2 +1 =0,$
 and
 $\xi^2_3\xi^3_2 - (\xi^2_2)^2 +1 =0.$
 Hence that case is impossible.
\\Case 3. There exists
$\varphi \in \text{Aut } \mathfrak{g}$ such that $v= \varphi(X_+ -X_-).$
Then we may suppose that $v=X_+ -X_-.$
Now instead of the basis $(H,X_+,X_-),$  we consider the basis
$(Y_1,Y_2,Y_3).$
 The matrix of $J$ in the basis
$(Y_1,Y_2,Y_3)$
has the form
\begin{equation*}
J_* =
\begin{pmatrix}
\eta^1_1&0&\eta_3^1\\
\eta^2_1&\lambda&\eta_3^2\\
\eta^3_1&0&\eta_3^2
\end{pmatrix}.
\end{equation*}
Then the 9 torsion equations
$*ij|k$ (the star is to underline that the new basis is in use)
for $J$ in that basis are:
\begin{eqnarray*}
*12|1
&\quad
 (\eta^3_1+ \eta^1_3 )\lambda - (\eta^3_1-\eta^1_3) \eta^1_1
=0,\\
*12|2
&\quad
 (\eta^1_1+\lambda) \eta^2_3 - \eta^2_1 \eta^3_1
=0,\\
*12|3
&\quad
 \eta^1_1 \lambda-1 +(\eta^3_1)^2 -(\eta^1_1 + \lambda)\eta^3_3
=0,\\
*13|1
&\quad
\eta^2_3 \eta^1_3 +\eta^2_1 \eta^1_1 +  \eta^2_3 \eta^3_1 - \eta^2_1 \eta^3_3
=0,\\
*13|2
&\quad
\eta^1_1 \lambda +1 + (\eta^2_1)^2 + (\eta^2_3)^2+   \eta^3_1 \eta^1_3 - (\eta^1_1
-\lambda)\eta^3_3
=0,\\
*13|3
&\quad
\eta^2_3 \eta^1_1 -\eta^2_1 (\eta^1_3 + \eta^3_1) - \eta^2_3 \eta^3_3
=0,\\
*23|1
&\quad
\eta^1_1 \lambda +1 - (\eta^1_3)^2 +  (\eta^1_1 -\lambda)\eta^3_3
=0,\\
*23|2
&\quad
\eta^2_3 \eta^1_3  - (\eta^3_3 +\lambda)\eta^2_1
=0,\\
*23|3
&\quad
(\eta^3_1 +\eta^1_3)\lambda  + (\eta^3_1 -\eta^1_3)\eta^3_3
=0.
\end{eqnarray*}
From $*12|1$  and $*23|3$,
\begin{equation}
\label{III}
\eta^1_1(\eta^3_1-\eta^1_3)
=-\eta^3_3(\eta^3_1-\eta^1_3).
\end{equation}
1) Suppose first that $\eta^3_1=\eta^1_3.$
Then $\lambda \eta^3_1=0.$
1.1) Consider the subcase $\eta^3_1 = 0.$
 $*13|1$ and $*13|3$ read resp.
 $(\eta^3_3-\eta^1_1)\eta^2_1=0,
 (\eta^3_3-\eta^1_1)\eta^2_3=0.$
 Suppose $\eta^3_3\neq \eta^1_1.$ Then
 $\eta^2_1= \eta^2_3 =0,$ and
 $*13|2$ gives $\eta^1_1 \lambda +1 = (\eta^1_1 -\lambda ) \eta^3_3,$
 which implies $\eta^3_3=0$  by  $*23|1.$
 As $*12|3$ then reads $1=0,$
 this case $\eta^3_3\neq \eta^1_1$ is not possible.
 Now, the case
  $\eta^3_3 = \eta^1_1$ is not possible either since then
  $*23|1$ would read $(\eta^1_1)^2 + 1 =0.$
 We conclude that the subcase 1.1) is not possible.
 Hence we are in the subcase
1.2)  $\eta^3_1 \neq  0.$ Then  $\lambda=0.$
From $*13|2$, $\eta^3_3\eta^1_1 \neq 0.$
Then $*23|1$ yields $\eta^3_3= \frac{-1+(\eta^3_1)^2}{\eta^1_1}$  and $*13|2$
reads $(\eta^2_1)^2 +(\eta^2_3)^2+2=0.$  This subcase 1.2) is not possible either.
Hence case 1) is not possible,
and  we are necessarily in the case 2)
$\eta^3_1 \neq \eta^1_3.$
From (\ref{III}), $\eta^3_3=-\eta^1_1.$
Then $*13|2$ reads
$ (\eta^1_1)^2 + (\eta^2_1)^2+   (\eta^2_3)^2 +1 + \eta^3_1 \eta^1_3 =0$
hence $\eta^3_1 \neq 0$ and
$\eta^1_3= -\frac{
(\eta^1_1)^2 + (\eta^2_1)^2+   (\eta^2_3)^2 +1}{ \eta^3_1}.$
From $*12|2,$
$\eta^2_1=\frac{\eta^2_3(\eta^1_1+\lambda)}{\eta^3_1}.$
Then $*23|2$ reads
$\eta^2_3(((\eta^2_3)^2 +\lambda^2+1)(\eta^3_1)^2+(\eta^1_1+\lambda)^2(\eta^2_3)^2)=0,$
i.e. $\eta^2_3=0,$
which implies
$\eta^2_1=0.$
Now $*12|1$ reads
$\lambda(1+(\eta^1_1)^2 -(\eta^3_1)^2)=-\eta^1_1(1+(\eta^1_1)^2 +(\eta^3_1)^2).$
The subcase $\eta^1_1 \neq 0$ is not possible
since then
 $*12|3$ would yield
 $\lambda = -\frac{(\eta^1_1)^2+(\eta^3_1)^2-1}{2\eta^1_1}$ and
 $*12|1$ would read
 $((\eta^1_1)^2+(\eta^3_1+1)^2)
 ((\eta^1_1)^2+(\eta^3_1-1)^2)=0.$
 Hence
$\eta^1_1 = 0.$ Then $*12|3$ reads $(\xi^3_1)^2=1.$
The condition $(\xi^3_1)^2=1$ now implies the vanishing of all the torsion equations.
In that case
\begin{equation*}
J_* =
\begin{pmatrix}
0&0&-\varepsilon\\
0&\lambda&0\\
\varepsilon&0&0
\end{pmatrix}
\quad,\quad \varepsilon = \pm 1.
\end{equation*}
Then
in the basis $(H,X_+,X_-)$
\begin{equation*}
J =
\begin{pmatrix}
0&-\frac{\varepsilon}{2}&-\frac{\varepsilon}{2}\\
\varepsilon&\frac{\lambda}{2}&-\frac{\lambda}{2}\\
\varepsilon&-\frac{\lambda}{2}&\frac{\lambda}{2}
\end{pmatrix}
\end{equation*}
The cases $\varepsilon =\pm 1$ are equivalent under $\Psi_0.$
\end{proof}
\begin{remark}
\rm
Recall that a
rank $r$ ($r\geqslant 1$)
$CR$-structure on a real Lie algebra
$\mathfrak{g}$ can be defined  (\cite{gigante})
as $(\mathfrak{p},J_\mathfrak{p})$ where
 $\mathfrak{p}$ is some
$2r$-dimensional vector subspace of
$\mathfrak{g}$ and
 $J_\mathfrak{p}\, : \, \mathfrak{p} \rightarrow  \mathfrak{p} $ is a linear map such that
 (a): $J_\mathfrak{p}^2=-1,$
 (b): $[X,Y] -[  J_\mathfrak{p}X,  J_\mathfrak{p}Y] \in \mathfrak{p} \quad \forall X,Y \in
 \mathfrak{p},$
 (c):
(\ref{Jnotorsion}) holds for
 $J_\mathfrak{p}$  for all $X,Y \in  \mathfrak{p}.$
 Then clearly
 $J_*(\lambda)$
is an extension of a $CR$-structure.
\end{remark}
\section{Case of
 $\mathfrak{sl}(2,\Rmath) \times  \mathfrak{sl}(2,\Rmath).$}
 We consider the basis $(Y_1^{(1)}, Y_2^{(1)}, Y_3^{(1)}, Y_1^{(2)}, Y_2^{(2)}, Y_3^{(2)})$
 of $\mathfrak{sl}(2,\Rmath) \times  \mathfrak{sl}(2,\Rmath),$
 with the upper index referring to the first or second factor.
 The automorphisms
 of $\mathfrak{sl}(2,\Rmath) \times  \mathfrak{sl}(2,\Rmath)$
 fall into 2 kinds: the first kind is comprised of
 the $\text{diag}(\Phi_1,\Phi_2)$,
  $\Phi_1,\Phi_2 \in \text{Aut} \,\mathfrak{sl}(2,\Rmath),$ and the second kind
 is comprised of the
 the $\Gamma \circ \text{diag}(\Phi_1,\Phi_2)$, with $\Gamma$ the switch between the two factors
 of $\mathfrak{sl}(2,\Rmath) \times  \mathfrak{sl}(2,\Rmath).$
\begin{lemma}
Any integrable complex structure $J$ on
 $\mathfrak{sl}(2,\Rmath) \times  \mathfrak{sl}(2,\Rmath)$
 is equivalent under some first kind automorphism
 to the endomorphism given in the basis
 $(Y_1^{(1)}, Y_2^{(1)}, Y_3^{(1)}, Y_1^{(2)}, Y_2^{(2)}, Y_3^{(2)})$
 by the matrix
 \begin{equation}
 \label{sl2xsl2bis}
\tilde{J}_*(\xi^2_2,\xi^2_5)=
\begin{pmatrix}
    0&0&-1& 0&0&0\\
    0&\xi^2_2&0& 0&\xi^2_5&0\\
    1&0&0&0&0&0 \\
    0&0&0&0&0&-1 \\
    0&-\frac{(\xi^2_2)^2+1}{\xi^2_5}&0&0&-\xi^2_2&0 \\
    0&0&0&1&0&0
    \end{pmatrix}
    , \;
    \xi^2_2, \xi^2_5 \in \Rmath
    \, , \,
    \xi^2_5 \neq 0    .
 \end{equation}
$\tilde{J}_*(\xi^2_2,\xi^2_5)$
is equivalent to
$\tilde{J}_*({\xi^{\prime}}^2_2,{\xi^{\prime}}^2_5)$
under some first (resp.  second)  kind automorphism  if and only if
$ {\xi^{\prime}}^2_2=   \xi^2_2,$
$ {\xi^{\prime}}^2_5=   \xi^2_5$
(resp.
$ {\xi^{\prime}}^2_2=  - \xi^2_2,$
$ {\xi^{\prime}}^2_5=
    -\frac{(\xi^2_2)^2+1}{\xi^2_5}$).
\end{lemma}
\begin{proof}
Let $J = (\xi^i_j)_{1 \leqslant i,j \leqslant 6}=
 \begin{pmatrix} J_1&J_2\\J_3&J_4 \end{pmatrix},$ ($J_1,J_2,J_3,J_4$  $3\times 3$ blocks),
an integrable complex structure
in the basis $(Y^{(k)}_\ell).$
From lemma \ref{sl2notorsion}, with some first kind automorphism, one may suppose
 $J_1=\begin{pmatrix} 0&0&-1\\ 0&\xi^2_2&0\\ 1&0&0 \end{pmatrix},$
 $J_4=\begin{pmatrix} 0&0&-1\\0&\xi^5_5&0\\ 1&0&0 \end{pmatrix}.$
 As $Tr(J)=0,$ $\xi^5_5=-\xi^2_2.$
 Then one is led to
 (\ref{sl2xsl2bis}) and the result follows (see \cite{companionarchive}, CSsl22.red and its output).

\end{proof}
\begin{remark}
\rm
The complex subalgebra $\mathfrak{m}$ associated  to
$\tilde{J}_*(\xi^2_2,\xi^2_5)$
has basis $\tilde{Y}^{(1)}_1 ={Y}^{(1)}_1   - i {Y}^{(1)}_3,$
$\tilde{Y}^{(2)}_1 ={Y}^{(2)}_1   - i {Y}^{(2)}_3,$
$\tilde{Y}^{(2)}_2 =-i\xi^2_5 {Y}^{(1)}_2   +(1+ i\xi^2_2) {Y}^{(2)}_2.$
The complexification
  $\mathfrak{sl}(2) \times  \mathfrak{sl}(2)$
 of $\mathfrak{sl}(2,\Rmath) \times  \mathfrak{sl}(2,\Rmath)$
 has weight spaces decomposition with respect to the Cartan subalgeba
 $\mathfrak{h}=\Cmath  {Y}^{(1)}_2 \oplus   \Cmath  {Y}^{(2)}_2 :$
 $$\mathfrak{h}\oplus  \Cmath ({Y}^{(1)}_1   + i {Y}^{(1)}_3)
 \oplus  \Cmath ({Y}^{(2)}_1   + i {Y}^{(2)}_3)
\oplus \Cmath \tilde{Y}^{(1)}_1
\oplus \Cmath \tilde{Y}^{(2)}_1.$$
Then
$\mathfrak{m}= (\mathfrak{h} \cap \mathfrak{m}) \oplus \Cmath \tilde{Y}^{(1)}_1
\oplus \Cmath \tilde{Y}^{(2)}_1$ with
$\mathfrak{h} \cap \mathfrak{m} =
 \Cmath \tilde{Y}^{(2)}_2,$ which is a special case of the general fact proved in \cite{snow}
 that any complex (integrable) structure on a reductive Lie group of class I is regular.
\end{remark}

\section{Case of $\mathfrak{n}$.}

Let   $\mathfrak{n}$  the real 3-dimensional Heisenberg Lie algebra with
basis $(x_1,x_2,x_3)$ and
commutation relations $[x_1,x_2]=x_3.$
\begin{lemma}
\label{lemma2}
 Let  $J \, : \,
 \mathfrak{n}
\rightarrow
 \mathfrak{n}
$ any linear map.
 $J$ has zero torsion
 if and only if it is equivalent to
one of the endomorphisms defined  in the basis $(x_1,x_2,x_3)$ by:
\begin{eqnarray}
\label{S}
(i) & &
S(\xi^3_3)=
\begin{pmatrix}
    0&-1&0\\
    1&0&0\\
    0&0&\xi^3_3
    \end{pmatrix}
    \quad, \quad \xi^3_3 \in \Rmath
    \\
\label{D}
(ii) & &
D(\xi^1_1)=
\begin{pmatrix}
    \xi^1_1&0&0\\
    0&\xi^1_1&0\\
    0&0&\frac{(\xi^1_1)^2-1}{2\xi^1_1}
    \end{pmatrix}  \quad , \quad \xi^1_1 \in \Rmath, \; \xi^1_1 \neq 0
\\
\label{T}
(iii) & &
T(a,b)=
\begin{pmatrix}
    0&-ab&0\\
    1&b&0\\
    0&0&\frac{ab-1}{b}
    \end{pmatrix}  \quad , \quad a, b \in \Rmath, \; b \neq 0
\end{eqnarray}
Any two distinct endomorphisms in the preceding list are non equivalent.
$T(a,b)$ is equivalent to
\begin{equation}
T^{\prime}(a,b)
=\begin{pmatrix}
    b&-b&0\\
    a&0&0\\
    0&0&\frac{ab-1}{b}
    \end{pmatrix}
\end{equation}
\end{lemma}
\begin{proof}
Let $J = (\xi^i_j)_{1 \leqslant i,j \leqslant 3}$
in the basis $(x_1,x_2,x_3).$
The 9 torsion equations are:
\begin{eqnarray*}
12|1
&\quad
\xi^1_3 (\xi^2_2 + \xi^1_1) =0,\\
12|2
&\quad
\xi^2_3 (\xi^2_2 + \xi^1_1) =0,\\
12|3
&\quad
 \xi^3_3 (\xi^2_2 +  \xi^1_1)-  \xi^2_2 \xi^1_1 +  \xi^2_1 \xi^1_2  +1 =0,\\
13|1
&\quad
\xi^2_3 \xi^1_3=0,\\
13|2
&\quad
(\xi^2_3)^2=0,\\
13|3
&\quad
\xi^2_3 ( \xi^3_3 - \xi^1_1) + \xi^1_2 \xi^1_3 =0,\\
23|1
&\quad
(\xi^1_3)^2=0,\\
23|2
&\quad
\xi^2_3 \xi^1_3=0,\\
23|3
&\quad
\xi^1_3 ( \xi^2_2 - \xi^3_3) - \xi^2_3 \xi^1_2 =0.
\end{eqnarray*}
Hence $\xi^1_3=\xi^2_3=0$ , and we are left only with equation
$12|3$ which reads
\begin{equation}
\xi^3_3\; Tr{(A)} =\det{(A)} -1
\end{equation}
where $A = \left(\begin{smallmatrix} \xi^1_1&\xi^1_2\\\xi^2_1&\xi^2_2             \end{smallmatrix}\right)      .$
Suppose first $Tr{(A)}=0.$ Then $A^2=-I,$ so that $A$ is similar over $\Cmath$, hence over $\Rmath$, to
 $ \left(\begin{smallmatrix} 0&-1 \\ 1&0            \end{smallmatrix}\right).$
 Hence $J \equiv \left(
\begin{smallmatrix}
    0&-1&0\\
    1&0&0\\
    *&*&\xi^3_3
    \end{smallmatrix} \right)$. Now, since $\xi^3_3$ does not belong to the spectrum of
 $ \left(\begin{smallmatrix} 0&-1 \\ 1&0            \end{smallmatrix}\right),$
    taking the automorphism
$ \left(\begin{smallmatrix} 1&0&0\\ 0&1&0\\ \alpha&\beta&1            \end{smallmatrix}\right)      $
of $\mathfrak{n}$
 for suitable $\alpha, \beta \in \Rmath,$ one gets
 $J \equiv S(    \xi^3_3  ).$
Suppose now $Tr{(A)}\neq 0.$ Then $\xi^3_3= \frac{\det{(A)}-1}{Tr{(A)}}.$ If $A$ is a scalar matrix,
i.e.
$A =  \xi^1_1 I,$ then
 $J = \left(
\begin{smallmatrix}
    \xi^1_1&0&0\\
    0&\xi^1_1&0\\
    *&*&\frac{(\xi^1_1)^2-1}{2\xi^1_1}
    \end{smallmatrix} \right)
    \equiv   D(\xi^1_1).$
    If $A$ is not a scalar matrix, then $A$ is similar to
 $ \left(\begin{smallmatrix} 0&-ab \\ 1&b            \end{smallmatrix}\right)$
 for some $a,b \in  \Rmath,$ and $b \neq 0$ from the trace.
 Then
    $J \equiv   T(a,b).$
    Finally, $T'(a,b) \equiv T(a,b)$
    since the matrices
$\left(\begin{smallmatrix} 0&-ab\\1&b \end{smallmatrix}\right)$
and
$\left(\begin{smallmatrix} b&-b\\a&0 \end{smallmatrix}\right)$
are similar for they have the same spectrum and are no scalar matrices.
\end{proof}
\begin{remark}
\rm
$S(\xi^3_3)$ is an extension of a rank 1 $CR$-structure, however
\linebreak[4]
$D(\xi^1_1),T(a,b)$ are not.

\end{remark}

\section{$CR$-structures on $\mathfrak{n}$.}
\begin{lemma}
(i)
Any linear map $J \, : \, \mathfrak{n} \rightarrow \mathfrak{n}$ which has zero torsion and is an extension of a rank 1 $CR$-structure
on  $\mathfrak{n}$ such that $\mathfrak{p}$ is nonabelian is equivalent to a unique
\begin{equation}
\begin{pmatrix}
0&-1&0\\
1&0&0\\
0&0&\xi^3_3
\end{pmatrix}
\; ,\; \xi^3_3 \in \Rmath.
\end{equation}
\\
(ii)
Any linear map $J \, : \, \mathfrak{n} \rightarrow \mathfrak{n}$ which is an extension of a rank 1 $CR$-structure
on  $\mathfrak{n}$ such that $\mathfrak{p}$ is abelian is equivalent to a unique
\begin{equation}
\begin{pmatrix}
\xi^1_1&0&0\\
0&0&1\\
0&-1&0
\end{pmatrix}
\; ,\; \xi^1_1 \in \Rmath.
\end{equation}
$J$ has  nonzero torsion.
\end{lemma}
\begin{proof}
For any nonzero $X \in \mathfrak{p},$ $(X,J_\mathfrak{p}X)$ is a basis of $\mathfrak{p}$.
In case (i),
$[X,J_\mathfrak{p}X]\neq 0,$ since  $\mathfrak{p}$ is non abelian. Then
$[X,J_\mathfrak{p}X]= \mu x_3$ , $\mu \neq 0$,  and $x_3 \not \in  \mathfrak{p}$
since otherwise
$\mathfrak{p}$ would be abelian. One may extend $J_\mathfrak{p}$ to
$\mathfrak{n}$ in the basis
$(X,J_\mathfrak{p}X, \mu x_3)$
as
\begin{equation}
J=\begin{pmatrix}
0&-1&\xi^1_3\\
1&0&\xi^2_3\\
0&0&\xi^3_3
\end{pmatrix}
\end{equation}
and $J$ has zero torsion only if $\xi^1_3=\xi^2_3=0.$
In case (ii), necessarily $x_3 \in \mathfrak{p}$ since $\mathfrak{p}$ is abelian.
Hence $(x_3, J_\mathfrak{p}x_3)$ is a basis for $\mathfrak{p}.$
Take any linear extension $J$ of $J_\mathfrak{p}$ to $\mathfrak{n}$.
There exists some eigenvector
$y_1\neq 0$  of $J$ associated to some eigenvalue $\xi^1_1 \in \Rmath.$
Then $y_1 \not \in \mathfrak{p},$
which implies $[y_1,Jx_3] \neq 0,$ for otherwise
$y_1$ would commute to the whole of $\mathfrak{n}$ and then be some multiple of $x_3 \in \mathfrak{p}.$
Hence
$[y_1,Jx_3]= \lambda x_3 $, $\lambda \neq 0,$ and dividing $y_1$ by $\lambda$ one may suppose
$\lambda=1.$ In the  basis $y_1, y_2= Jx_3, y_3=x_3$
one has
\begin{equation}
J=\begin{pmatrix}
\xi^1_1&0&0\\
0&0&1\\
0&-1&0
\end{pmatrix}
\end{equation}
and (ii) follows.
\end{proof}

\section{Complex structures on $\mathfrak{n}\times \mathfrak{n}$.}
We will use for commutation relations
$[x_1,x_2]=x_5,
[x_3,x_4]=x_6.$
The automorphisms of
 $\mathfrak{n}\times \mathfrak{n}$ fall into 2 kinds. The first kind is comprised
of the matrices
\begin{multline}
\Phi = \left(\begin{array}{cccc|cc}
b^1_1&b^1_2&0&0&0&0\\
b^2_1&b^2_2&0&0&0&0\\
0&0&b^3_3&b^3_4&0&0\\
0&0&b^4_3&b^4_4&0&0\\
\hline
b^5_1&b^5_2&b^5_3&b^5_4 & b^1_1 b^2_2 -b^1_2 b^2_1& 0\\
b^6_1&b^6_2&b^6_3&b^6_4 & 0 &b^3_3 b^4_4 -b^3_4 b^4_3
       \end{array}
       \right)
       \quad , \quad \\
(b^1_1 b^2_2 -b^1_2 b^2_1)
(b^3_3 b^4_4 -b^3_4 b^4_3) \neq 0.
\end{multline}
The second kind ones are $\Psi = \Theta \Phi $ where $\Phi$ is
first kind  and
\begin{equation}
\Theta = \left(\begin{array}{cccc|cc}
0&0&1&0&0&0\\
0&0&0&1&0&0\\
1&0&0&0&0&0\\
0&1&0&0&0&0\\
\hline
0&0&0&0&0&1\\
0&0&0&0&1&0
       \end{array}
       \right).
\end{equation}

\begin{lemma}
Any integrable complex structure $J$ on
 $\mathfrak{n}\times \mathfrak{n}$
 is equivalent under some first kind automorphism to one of the following:
 \begin{equation}
 \label{case1}
(i) \quad
\tilde{S}_\varepsilon(\xi^5_5)=
\begin{pmatrix}
    0&-1&0& 0&0&0\\
    1&0&0& 0&0&0\\
    0&0&0&-1&0&0 \\
    0&0&1&0&0&0 \\
    0&0&0&0&\xi^5_5&-\varepsilon((\xi^5_5)^2+1) \\
    0&0&0&0&\varepsilon&-\xi^5_5
    \end{pmatrix}
    , \;
    \varepsilon =\pm 1
    \, , \,
    \xi^5_5 \in \Rmath
    .
 \end{equation}
$\tilde{S}_{\varepsilon^{\prime}} ({\xi^{\prime}}^5_5)$
is equivalent to
$\tilde{S}_\varepsilon(\xi^5_5)$
($\varepsilon,\varepsilon^{\prime}= \pm1; {\xi^{\prime}}^5_5,      \xi^5_5 \in \Rmath$)
under some first (resp.  second)  kind automorphism  if and only if
$\varepsilon^{\prime}= \varepsilon , \, {\xi^{\prime}}^5_5=   \xi^5_5$
(resp. $\varepsilon^{\prime}= -\varepsilon ,\,  {\xi^{\prime}}^5_5=  - \xi^5_5$ ).

 \begin{multline}
(ii)    \quad
 \label{case2}
\tilde{D}(\xi^1_1)=
\begin{pmatrix}
    \xi^1_1&0&-((\xi^1_1)^2+1)& 0&0&0\\
    0&\xi^1_1&0&-((\xi^1_1)^2+1)& 0&0\\
    1&0&-\xi^1_1&0&0&0 \\
    0&1&0&-\xi^1_1&0&0 \\
    0&0&0&0&\frac{(\xi^1_1)^2-1}{2\xi^1_1}&-\frac{((\xi^1_1)^2+1)^2}{2\xi^1_1} \\
    0&0&0&0&\frac{1}{2\xi^1_1}&\frac{1-(\xi^1_1)^2}{2\xi^1_1}
    \end{pmatrix}
    ,\\
    \; \xi^1_1 \in \Rmath\setminus \{0\}.
 \end{multline}
$\tilde{D}({\xi^{\prime}}^1_1)$
is equivalent to
$\tilde{D}(\xi^1_1)$
($ {\xi^{\prime}}^1_1,      \xi^1_1 \in \Rmath$)
under some first (resp.  second)  kind automorphism  if and only if
$ {\xi^{\prime}}^1_1=   \xi^1_1$
(resp. $ {\xi^{\prime}}^1_1=  - \xi^1_1$ ).
 \begin{multline}
(iii)    \quad
 \label{case3}
\tilde{T}(\xi^3_3,\xi^4_3)= \\
\begin{pmatrix}
    0&-\xi^4_3\xi^3_3&-\xi^4_3\xi^3_3&\xi^4_3\xi^3_3-1&0&0\\
    1&-\xi^3_3&-\frac{(\xi^3_3)^2+1-\xi^4_3 \xi^3_3}{\xi^3_3}&\xi^3_3& 0&0\\
    0&\xi^3_3&\xi^3_3&-\xi^3_3& 0&0\\
    1&0&\xi^4_3&0&0&0 \\
    0&0&0&0&-\frac{\xi^4_3\xi^3_3-1}{\xi^3_3}&-\frac{(\xi^4_3\xi^3_3-2)\xi^4_3\xi^3_3+(\xi^3_3)^2+1}{(\xi^3_3)^2} \\
    0&0&0&0&1&\frac{\xi^4_3\xi^3_3-1}{\xi^3_3}
    \end{pmatrix},
    \\
     \xi^3_3 \in \Rmath\setminus \{0\} , \;
    \xi^4_3 \in \Rmath.
 \end{multline}
$\tilde{T}({\xi^{\prime}}^3_3,{\xi^{\prime}}^4_3)$
is equivalent to
$\tilde{T}(\xi^3_3,\xi^4_3)$
($
{\xi^{\prime}}^3_3,
     \xi^3_3 \in \Rmath\setminus \{0\} , \;
 {\xi^{\prime}}^4_3,
    \xi^4_3 \in \Rmath.$)
under some first (resp.  second)  kind automorphism  if and only if
$ {\xi^{\prime}}^3_3=   \xi^3_3, \, {\xi^{\prime}}^4_3=   \xi^4_3$
(resp.
$ {\xi^{\prime}}^3_3=  - \xi^3_3, \, {\xi^{\prime}}^4_3= -  \xi^4_3$).
\par
Finally, the cases (i),(ii), (iii) are mutually non equivalent, either under first
or second kind automorphism.
\end{lemma}
\begin{proof}
Let $J = (\xi^i_j)_{1 \leqslant i,j \leqslant 6}$ an integrable complex structure
in the basis $(x_k)_{1\leqslant k \leqslant 6}.$
Denote
$J_1=\left(\begin{smallmatrix}
    \xi^1_1&\xi^1_2\\
    \xi^2_1&\xi^2_2
    \end{smallmatrix} \right)
    $,
$J_2=\left(\begin{smallmatrix}
    \xi^1_3&\xi^1_4\\
    \xi^2_3&\xi^2_4
    \end{smallmatrix} \right)
    $,
$J_3=\left(\begin{smallmatrix}
    \xi^3_1&\xi^3_2\\
    \xi^4_1&\xi^4_2
    \end{smallmatrix} \right)
    $,
$J_4=\left(\begin{smallmatrix}
    \xi^3_3&\xi^3_4\\
    \xi^4_3&\xi^4_4
    \end{smallmatrix} \right)
    $.
Then
$J_1^*=\left(\begin{smallmatrix}
    \xi^1_1&\xi^1_2&\xi^1_5\\
    \xi^2_1&\xi^2_2&\xi^2_5\\
    \xi^5_1&\xi^5_2&\xi^5_5
    \end{smallmatrix} \right)
    $
    and
$J_3^*=\left(\begin{smallmatrix}
    \xi^3_1&\xi^3_4&\xi^3_6\\
    \xi^4_3&\xi^4_4&\xi^4_6\\
    \xi^6_3&\xi^6_4&\xi^6_6
    \end{smallmatrix} \right) $
    are zero torsion linear maps from $\mathfrak{n}$ to
    $\mathfrak{n},$ hence equivalent to  type
    (\ref{S}),    (\ref{D}) or     (\ref{T})     in lemma \ref{lemma2}.
    It can be checked that their being of different types would contradict with
    $J^2=-1.$
    Hence, modulo equivalence
under some  first kind automorphism,
we get 3 cases:
case 1:
$J_1^*=
\left(\begin{smallmatrix}
    0&-1&0\\
    1&0&0\\
    0&0&\xi^5_5
    \end{smallmatrix} \right)
    $,
$J_3^*=
\left(\begin{smallmatrix}
    0&-1&0\\
    1&0&0\\
    0&0&\xi^6_6
    \end{smallmatrix} \right) ;
    $
case 2: $J_1^*=D(\xi^1_1), J_3^* =D(\xi^3_3)$,
($\xi^1_1, \,\xi^3_3  \neq 0$);
case 3:
$J_1^*=\left(\begin{smallmatrix}
    0&\xi^1_2&0\\
    1&\xi^2_2&0\\
    0&0&\xi^5_5
    \end{smallmatrix} \right)
    $,
$J_3^*=\left(\begin{smallmatrix}
    \xi^3_3&-\xi^3_3&0\\
    \xi^4_3&0&0\\
    0&0&\xi^6_6
    \end{smallmatrix} \right)
    $,
($\xi^2_2, \,\xi^3_3  \neq 0$).
Case 1 (resp. 2, 3) leads to
(\ref{case1}) (resp.
(\ref{case2}),
(\ref{case3}))
(see  \cite{companionarchive}, programs
"n2case1.red", "n2case2.red", "n2case3.red", and their outputs.)
The assertion about equivalence in cases 1,2 are readily proved, as is equivalence under
some first kind automorphism in case 3 and the nonequivalence of the 3 types.
Consider now $\Theta \tilde{T}(\xi^3_3,\xi^4_3) \Theta^{-1}.$
It is equivalent under some first kind automorphism to some
$\tilde{T}(\eta^3_3,\eta^4_3).$  That implies that the matrices
$\left(\begin{smallmatrix}
    \xi^3_3&-\xi^3_3\\
    \xi^4_3&0
    \end{smallmatrix} \right)    $,
$\left(\begin{smallmatrix}
    0&-\eta^4_3\eta^3_3\\
    1&-\xi^3_3
    \end{smallmatrix} \right)    $
    are similar, which amounts to their having same trace and same determinant, i.e.
    $\eta ^3_3= -\xi^3_3, \eta^4_3=-\xi^4_3.$
 As
$\tilde{T}({\xi^{\prime}}^3_3, {\xi^{\prime}}^4_3)$ is equivalent
to
$ \tilde{T}(\xi^3_3,\xi^4_3)$
under some second kind automorphism if and only if
it is equivalent to
$\Theta \tilde{T}(\xi^3_3,\xi^4_3) \Theta^{-1}$ under some first kind automorphism,
the assertion about second kind equivalence in case 3 follows.
\end{proof}
\begin{remark}
\rm
In case 3, had we used
$J_3^*=\left(\begin{smallmatrix}
    0&\xi^3_4&0\\
    1&\xi^4_4&0\\
    0&0&\xi^6_6
    \end{smallmatrix} \right)
    $,
    then we would have to separate further into 2 subcases:
    subcase $\xi^1_2 \neq 0$:
    \begin{multline*}
\tilde{T}(\xi^1_2,\xi^2_2)=
\begin{pmatrix}
    0&\xi^1_2&-\frac{\xi^2_2}{\xi^1_2}&-(\xi^1_2+1)&0&0\\
    1&\xi^2_2& \frac{\xi^1_2+1}{\xi^1_2}&-\xi^2_2& 0&0\\
    0&-\xi^1_2&0&\xi^1_2& 0&0\\
    1&\xi^2_2&1&-\xi^2_2& 0&0\\
    0&0&0&0&-\frac{\xi^1_2+1}{\xi^2_2}&-\frac{(\xi^2_2)^2+(\xi^1_2+1)^2}{\xi^2_2\xi^1_2}\\
    0&0&0&0&\frac{\xi^1_2}{\xi^2_2}&\frac{\xi^1_2+1}{\xi^2_2}
    \end{pmatrix},                \quad
     \xi^1_2 \xi^2_2 \neq 0  \, ;
    \end{multline*}
     subcase $\xi^1_2 = 0$:
    \begin{multline*}
\tilde{T}(\xi^2_2)=
\begin{pmatrix}
    0&0&-1&0&0&0\\
    1&\xi^2_2&0&1& 0&0\\
    1&0&0&0& 0&0\\
    -\xi^2_2&-((\xi^2_2)^2+1)&1&-\xi^2_2& 0&0\\
    0&0&0&0&-\frac{1}{\xi^2_2}&\frac{1}{\xi^2_2}\\
    0&0&0&0&-\frac{(\xi^2_2)^2+1}{\xi^2_2}&\frac{1}{\xi^2_2}
    \end{pmatrix},                \quad
     \xi^2_2 \neq 0  \, .
    \end{multline*}
\end{remark}
\begin{remark}
\rm
$\tilde{S}_\varepsilon(\xi^5_5)$
is abelian.
\end{remark}

\begin{remark}
\rm
If one looks for zero torsion linear maps instead of complex structures, then
$J_1^*$ and $J_3^*$ may be of different types.
\end{remark}


\end{document}